\documentclass[graybox]{svmult}
\usepackage{mathptmx}       
\usepackage{helvet}         
\usepackage{courier}        
\usepackage{type1cm}        
\usepackage{graphicx}       

\usepackage{array,colortbl}
\usepackage{amsmath,amsfonts,amssymb,bm} 
\usepackage[mathx]{mathabx}
\DeclareFontFamily{U}{mathx}{\hyphenchar\font45}
\DeclareFontShape{U}{mathx}{m}{n}{
      <5> <6> <7> <8> <9> <10>
      <10.95> <12> <14.4> <17.28> <20.74> <24.88>
      mathx10
      }{}
\DeclareSymbolFont{mathx}{U}{mathx}{m}{n}
\DeclareFontSubstitution{U}{mathx}{m}{n}
\DeclareMathAccent{\widecheck}      {0}{mathx}{"71}

\DeclareSymbolFont{bbold}{U}{bbold}{m}{n}
\DeclareSymbolFontAlphabet{\mathbbold}{bbold}
\newcommand{\ind}{\mathbbold{1}}

\DeclareMathOperator{\var}{var}

\usepackage{microtype} 

\usepackage[colorlinks=true,linkcolor=black,citecolor=black,urlcolor=black]{hyperref}
\usepackage{bookmark}
\makeatletter 
  \providecommand*{\toclevel@author}{999}
  \providecommand*{\toclevel@title}{0}
\makeatother

\usepackage{xspace} 
\DeclareMathOperator{\Ber}{Ber}

\begin{document}
\newcommand{\bsa}{\boldsymbol{a}}    
\newcommand{\bsb}{\boldsymbol{b}}    
\newcommand{\bsX}{\boldsymbol{X}}    
\newcommand{\dnorm}{\mathcal{N}}
\newcommand{\tmu}{\tilde{\mu}}
\newcommand{\hmu}{\hat{\mu}}
\newcommand{\hp}{\hat{p}}
\newcommand{\e}{\mathbb{E}}
\newcommand{\abs}[1]{\left|#1\right|}
\newcommand{\meanMCB}{\texttt{meanMCBer\_g}\xspace}
\newcommand{\Cheb}{\text{Cheb}}
\newcommand{\CLT}{\text{CLT}}
\newcommand{\Hoeff}{\text{Hoeff}}

\spdefaulttheorem{algo}{Algorithm}{\upshape \bfseries}{\upshape}
\title*{Guaranteed Monte Carlo Methods for Bernoulli Random Variables}
\author{Lan Jiang \and Fred J.\ Hickernell}
\institute{
Lan Jiang \and Fred J. Hickernell 
\at Illinois Institute of Technology, 10 E. 32nd st, Chicago, IL, USA 
\email{ljiang14@hawk.iit.edu}, \email{hickernell@iit.edu}
}
\maketitle

\abstract{Simple 
Monte Carlo is a versatile computational method with a convergence rate of $O(n^{-1/2})$. It can be used to estimate the means of random variables whose distributions are unknown. Bernoulli random variables, $Y$, are widely used to model success (failure) of complex systems.   Here $Y=1$ denotes a success (failure), and  $p=\e(Y)$ denotes the probability of that success (failure).  Another application of Bernoulli random variables is $Y=\ind_{R}(\bsX)$, where then $p$ is the probability of $\bsX$ lying in the region $R$. This article explores how estimate $p$ to a prescribed absolute error tolerance, $\varepsilon$, with a high level of confidence, $1-\alpha$.  The proposed algorithm automatically determines the number of samples of $Y$ needed to reach the prescribed error tolerance with the specified confidence level by using Hoeffding's inequality. The algorithm described here has been implemented in MATLAB and is part of the Guaranteed Automatic Integration Library (GAIL).}

\section{Introduction}\label{section1}
Monte Carlo is a widely simulation method for approximating means of random variables, quantiles, integrals, and optima. In the case of estimating the mean of a random variable $Y$, the Strong Law of Large Numbers ensures that the sample mean converges to the true solution almost surely, i.e.: $\lim_{n \to \infty} \hmu_n =\mu \text{ a.s.}$ \cite[Theorem 20.1]{JP04}.  The Central Limit Theorem (CLT) provides a way to construct an approximate confidence interval for the $\mu$ in terms of the sample mean assuming a known variance of $Y$, however, this is not a finite sample result.  A conservative fixed-width confidence interval under the assumption of a known bound on the kurtosis is provided by \cite{HJLO12}.

Here we construct a conservative fixed-width confidence interval for Bernoulli random variables, $Y$. Here the mean is the probability of success, i.e,  $p:= \e(Y)=\Pr(Y=1)$.  This distribution is denoted by $\Ber(p)$.  Possible applications include the probability of bankruptcy or a power failure, where the process governing $Y$ may have a complex form.  This means that we may be able to to generate independent and identically distributed (IID) $Y_i$, but not have a simple formula for computing $p$ analytically. 

This paper presents an automatic simple Monte Carlo method for constructing a fixed-width (specified error tolerance) confidence interval for $p$ with a \textit{guaranteed} confidence level.  That is, given a tolerance, $\varepsilon$, and confidence level, $1-\alpha$, the algorithm determines the sample size, $n$, to compute a sample mean, $\hp_n$, that satisfies the condition $\Pr(\abs{p-\hp_n}\leq \varepsilon) \geq 1-\alpha$.  Moreover, there is an explicit formula for the computational cost of the algorithm in terms of $\alpha$ and $\varepsilon$. A publicly available implementation of our algorithm, called \meanMCB, is part of the next release of the Guaranteed Automatic Integration Library \cite{GAIL_1_3}.

Before presenting our new ideas, we review some of the existing literature.  While the CLT is often used for constructing confidence intervals, it relies on unjustified approximations.  We would like to have confidence intervals backed by theorems.

As mentioned above, \cite{HJLO12} presents a  reliable fixed-width confidence interval for evaluating the mean of an arbitrary random variable via Monte Carlo sampling based on the assumption that the kurtosis has a known upper bound. This algorithm uses the Cantelli's inequality to get a reliable upper bound on the variance and applies the Berry-Esseen inequality to determine the sample size needed to achieve desired confidence interval width and confidence level. For the algorithm in \cite{HJLO12} the distribution of the random variable is arbitrary. 

Wald confidence interval \cite[Section 1.3.3]{Agresti02} is a commonly used one based on maximum likelihood estimate, unfortunately, it performs poorly when the sample size $n$ is small or the true $p$ is close to 0 or 1. Agresti \cite[Section 1.4.2]{Agresti02} suggested constructing confidence intervals for binomial proportion by adding a pseudo-count of $z_{\alpha/2}/2$ successes and failures.  Thus, the estimated mean would be $\tilde{p}_n =(n\hp_n+z_{\alpha/2}/2) / (n+z_{\alpha/2})$. This method is also called adjusted Wald interval or Wilson score interval, since it was first discussed by E. B. Wilson \cite{wilson27}. This method performs better than the Wald interval.  However, it is an approximate result and carries no guarantee.

Clopper and Pearson \cite{CP34} suggested a tail method to calculate the exact confidence interval for a given sample size $n$ and uncertainty level $\alpha$. Sterne \cite{sterne54}, Crow \cite{crow56}, Blyth and Still \cite{BS83} and Blaker \cite{Blaker00} proposed different ways to improve the exact confidence interval, however, all of them were only tested on small sample sizes, $n$.  Moreover, these authors did not suggest how to determine $n$ that gives a confidence interval with fixed half-width $\varepsilon$. 

An outline of this paper follows. Section \ref{section2} provides the key theorems and inequalities needed. Section \ref{section3} describes an algorithm, \meanMCB, that estimates the mean of Bernoulli random variables to a prescribed absolute error tolerance with guaranteed confidence level and the proof of its success. The computational cost of the algorithm is also derived. Section \ref{section4} provides a numerical example of \meanMCB and compares the computational cost to confidence intervals based on the Central Limit Theorem. The paper ends with the discussion of future work.

\section{Basic Theorems and Inequalities}\label{section2}

Confidence intervals for the mean of a random variable, $Y$, are based on the mean of IID samples of that random variable:
\begin{equation}
\hmu_n := \frac1n \sum_{i=1}^n Y_i, \qquad Y_1, Y_2, \ldots \text{ IID}.
\end{equation}
We review some theorems describing how close $\hmu_n$ must be to $\mu:=\e(Y)$.

\subsection{Chebyshev's Inequality}
Chebyshev's inequality may be used to construct a fixed-width confidence interval for $\mu$.  It makes relatively mild assumptions on the distribution of the random variable.

\begin{theorem}[Chebyshev's Inequality {\cite[6.1.c]{LB10}}] If $X$ is a random variable with mean $\mu$, then $\Pr(\abs{X-\mu} \ge \varepsilon) \le  \var(X)/\varepsilon^2$.
\end{theorem}
Choosing $X=\hmu_n$, noting that $\var(X) = \var(Y)/n$, and setting $\var(X)/\varepsilon^2=\alpha$ leads to the fixed-width confidence interval 
\[
\Pr(\abs{\hmu_{n_{\Cheb}}-\mu} \le \varepsilon) \ge 1- \alpha \qquad \text{for } n_{\Cheb}:= \left \lceil \frac{\var(Y)}{\alpha \varepsilon^2} \right \rceil,
\]
provided that $\var(Y)$ is known.  

For $Y \sim \Ber(p)$, we know that $\var(Y)=p(1-p) \le 1/4$.  Letting $\hp_n$ denote the sample mean of  Bernoulli random variables, we have the fixed-width confidence interval 
\begin{equation} \label{ChebBerCI}
\Pr(\abs{\hp_{n_{\Cheb}}-p} \le \varepsilon) \ge 1- \alpha \qquad \text{for } n_{\Cheb} := \left \lceil \frac{1}{4\alpha \varepsilon^2} \right \rceil, \quad Y \sim \Ber(p).
\end{equation}
The upper bound on $\var(Y)$ is used rather than the exact formula for $\var(Y)$ because $p$ is unknown.

The factor of $1/\alpha$ in $n_{\Cheb}$ makes this confidence interval quite costly. Fortunately, there are other options.

\subsection{Central Limit Theorem (CLT)}
The CLT describes how the distribution of $\hmu_n$ approaches a Gaussian distribution as $n \to \infty$.
\begin{theorem}[Central Limit Theorem {\cite[Theorem 21.1]{JP04}}] \label{clt} 
If $Y_1, \ldots, Y_n$ are IID with $\e(Y_i)=\mu$, then
$$
\frac{\hmu_n-\mu}{\sqrt{\var(Y)/n}} \to \dnorm(0,1) \quad \text{in distribution, as} \ \ n\to\infty.
$$
\end{theorem}
This theorem implies an approximate confidence interval, called a CLT confidence interval, of the form
\[
\Pr(\abs{\hmu_{n_{\CLT}}-\mu} \le \varepsilon) \approx 1- \alpha \qquad \text{for } n_{\CLT} := \left \lceil \frac{z_{\alpha/2}\var(Y)}{\alpha \varepsilon^2} \right \rceil,
\]
where $z_{\alpha/2}$ is the $1-\alpha/2$ quantile of the standard Gaussian distribution.  When $\var(Y)$ is unknown, it may be replaced by the sample variance. 
For Bernoulli random variables we use the upper bound on $\var(Y)$ to obtain
\begin{equation} \label{CLTBerCI}
\Pr(\abs{\hp_{n_{\CLT}}-p} \le \varepsilon) \gtrapprox 1- \alpha \qquad \text{for } n_{\CLT} := \left \lceil \frac{z_{\alpha/2}}{4\varepsilon^2} \right \rceil, \quad Y \sim \Ber(p).
\end{equation}

Since $z_{\alpha/2}$ satisfies the equation
\begin{align*}
\frac {\alpha}{2} &= \int_{z_{\alpha/2}}^{\infty} \frac{\E^{-x^2/2}}{\sqrt{2 \pi}}  \, \D x = \left . \frac{- \E^{-x^2/2}}{\sqrt{2 \pi} x} \right \rvert_{z_{\alpha/2}}^{\infty} - \int_{z_{\alpha/2}}^{\infty} \frac{\E^{-x^2/2}}{\sqrt{2 \pi} x^2}  \, \D x  \\
& = \frac{\E^{-z_{\alpha/2}^2/2}}{\sqrt{2 \pi} z_{\alpha/2}}  - \int_{z_{\alpha/2}}^{\infty} \frac{\E^{-x^2/2}}{\sqrt{2 \pi} x^2}  \, \D x, 
\end{align*}
it follows that $z_{\alpha/2} = o(\sqrt{\log(1/\alpha)})$ as $\alpha \to 0$.  This is a much slower increase than the $1/\alpha$ term in the Chebyshev confidence interval.

Unfortunately, the CLT is an asymptotic result and not guaranteed to hold for a finite $n$.  We need a different inequality to provide a finite sample result.

\subsection{Hoeffding's Inequality}

The conservative fixed-width confidence interval for general random variables constructed by  \cite{HJLO12} uses the Berry-Esseen Inequality \cite[Section 4.1]{LB10} and Cantelli's Inequality \cite[Section 6.1]{LB10}. However, in view of the particular form of the Bernoulli distribution, we may rely on inequalities that assume some bound on the random variable.  Here we use Hoeffding's Inequality \cite{H63}, which seems more suitable than, say, the Chernoff bound \cite{chernoff52}.  Below is the a special case of Hoeffding's Inequality for random variables lying in $[0,1]$.
\begin{theorem}[Hoeffding's Inequality {\cite{H63}}] \label{hoeff}
If $Y_1$, $Y_2$, $\cdots$, $Y_n$ are IID observations such that $\e(Y_i)=p$ and $0 \leq Y_i \leq 1$, define $\hp_n=\sum_{i=1}^n Y_i/n$. Then, for any $\varepsilon>0$, 
$$\Pr(\hp_n-p \geq \varepsilon) \leq e^{-2n\varepsilon^2},$$
$$\Pr(p-\hp_n \geq \varepsilon) \leq e^{-2n\varepsilon^2},$$
$$\Pr(\abs{\hp_n-p} \geq \varepsilon) \leq 2e^{-2n\varepsilon^2}.$$
\end{theorem}

\section{Guaranteed Fixed-Width Confidence Intervals for  the Mean of Bernoulli Random Variables}\label{section3}

Theorem \ref{hoeff} motivates the following algorithm for constructing fixed-width confidence intervals for an unknown $p$ in terms of IID samples of $Y\sim \Ber(p)$.
\begin{algo}[\meanMCB]\label{algabs}
Given an error tolerance $\varepsilon \geq 0$ and an uncertainty $\alpha \in (0,1)$ , generate
\begin{equation}\label{hoeffn}
n = n_{\Hoeff} := \left \lceil  \frac{\log(2/\alpha)}{2\varepsilon^2} \right \rceil
\end{equation}
IID Bernoulli random samples of $Y \sim \Ber(p)$, and use them to compute sample mean:
\begin{equation} \label{abserrp}
\hat{p}_n = \sum_{i =1}^n Y_i.
\end{equation}
Return $\hp_n$ as the answer.
\end{algo}
\begin{theorem}
Algorithm \ref{algabs} (\meanMCB) returns an answer that satisfies 
\begin{equation}\label{probabs}
\Pr(|\hat{p}_n -p| \leq \varepsilon) \geq 1-\alpha.
\end{equation}
at a computational cost of $n_{\Hoeff}$ samples.
\end{theorem}
\begin{proof}
The formula for the sample size in \eqref{hoeffn} may be used to show that
\begin{equation*}
n = \left \lceil \frac{\log(2/\alpha)}{2\varepsilon^2} \right \rceil \geq  \frac{\log(2/\alpha)}{2\varepsilon^2} \Rightarrow 
1- 2e^{-2n\varepsilon^2} \geq 1-\alpha.
\end{equation*}
Applying  Hoeffding's inequality in Theorem \ref{hoeff} leads directly to \eqref{probabs}. \qed
\end{proof}

\section{Numerical Examples}\label{section4}
Algorithm \ref{algabs} (\meanMCB) has been implemented in MATLAB \cite{MATLAB14}, using the name \meanMCB.  It will be part of the next release of the Guaranteed Automatic Integration Library (GAIL) \cite{GAIL_1_3}.  The GAIL library also includes automatic algorithms from  \cite{CDHHZ13} and \cite{HJLO12}.

\subsection{Demonstration for $Y\sim \Ber(p)$ for Various $p$ and $\varepsilon$}
To demonstrate its performance the algorithm \meanMCB was run for $500$ replications, each with a different $p$ and $\varepsilon$, and with a fixed confidence level $1-\alpha=95\%$. The logarithms of $p$ and $\varepsilon$ were chosen independently and uniformly, namely 
\[
\log_{10} p \sim \mathcal{U}[-3,-1], \qquad \log_{10} \varepsilon \sim \mathcal{U}[-5,-2].
\]
For each replication, the inputs $\varepsilon$ and $\alpha=5\%$ were provided to \meanMCB, along with a $\Ber(p)$ random number generator, and the answer $\hp_n$ was returned.  

Figure \ref{fig:abserrex} shows the ratio of the true error to the absolute error tolerance, $\abs{p-\hp_n}/\varepsilon$, for each of the $500$ replications, plotted against $p$.  All of the replications resulted in $\abs{p-\hp_n} \le \varepsilon$, which is better than guaranteed $95\%$ confidence level.  For some of these replications, \meanMCB was asked to exceed its sample budget of $10^{10}$, namely when $\varepsilon \le \sqrt{\log(2/0.05)} \times 10^{-5} = 3.69 \times 10^{-5}$.

  \begin{figure}[htbp]
    \centering
    \includegraphics[width=9cm]{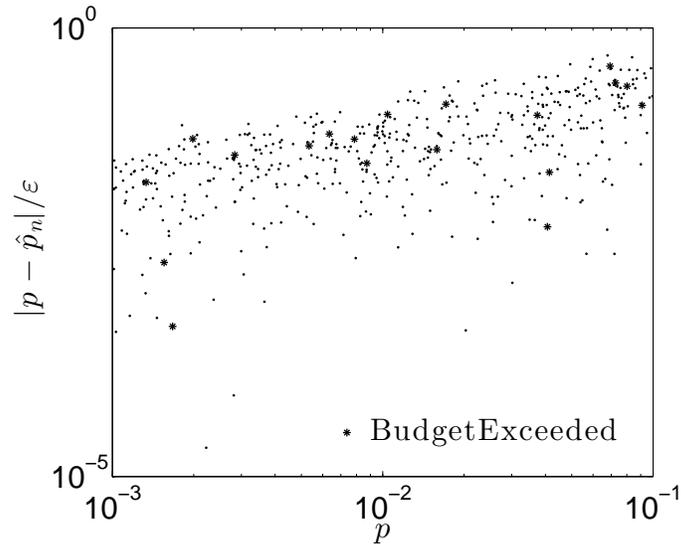} 
    \caption{Ratio of the actual absolute error to the error tolerance from \meanMCB versus $p$, for different random samples of $\Ber(p)$ random variables.}
    \label{fig:abserrex}
 \end{figure}
 
While it is encouraging to see that \meanMCB provides the correct answer in all cases, it is concerning that \meanMCB is rather conservative for small $p$.  This is due to the fact that the error of $\hp_n$ using $n$ samples is expected to be proportional to $\sqrt{\var(Y)/n}=\sqrt{p(1-p)/n}$.  Even though the error is small for small $p$, our algorithm does not take advantage of that fact.  To do so would require at least a loose lower bound on $p$ at the same time that the algorithm is trying to determine the sample size needed to estimate $p$ carefully.

\subsection{CLT \& Hoeffding's Inequality Confidence Interval Cost Comparison}
By using Hoeffding's inequality to construct guaranteed fixed-width confidence interval, we definitely incur additional cost compared to an approximate CLT confidence interval.  The ratio of this cost is 
\begin{equation}
\frac{n_{\Hoeff}}{n_{\CLT}} = \frac{\left \lceil \log(2/\alpha)/{2\varepsilon^2} \right \rceil}{\left \lceil{ \Phi^{-1}(1-\alpha/2)}/{4 \varepsilon^2}\right\rceil} \approx  \frac{2\log(2/\alpha)}{\Phi^{-1}(1-\alpha/2)}.
\end{equation}
This ratio essentially depends on the uncertainty level $\alpha$ and is plotted in Figure \ref{fig:ratiovsalpha}. For $\alpha$ between $0.01\%$ to $10\%$ this ratio is between 3.64 to 5.09, which we believe is a reasonable price to pay for the added certainty of $\meanMCB$.

  \begin{figure}[htbp]
    \centering
    \includegraphics[width=8cm]{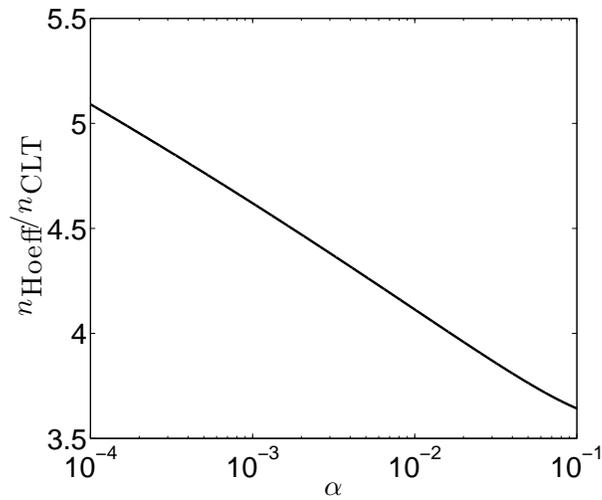} 
    \caption{The computational cost ratio of using Hoeffding's inequality and the CLT to construct a fixed-width confidence interval.}
    \label{fig:ratiovsalpha}
 \end{figure}

\section{Future Work}\label{section5}
Recent work by our collaborators and ourselves on automatic algorithms as described in \cite{CDHHZ13, HJLO12, HicJim16a, JimHic16a} is prompted by a desire to provide robust and practical Monte Carlo, quasi-Monte Carlo, and other numerical algorithms.  We want algorithms that require minimal input from practitioners and are theoretically justified.  This is also the motivation for \meanMCB that constructs fixed-width confidence intervals for Bernoulli random variables.  The algorithms in the above-mentioned references are available, or will soon be available, in the GAIL MATLAB toolbox \cite{GAIL_1_3}.

While Algorithm \ref{algabs} satisfies an absolute error tolerance, an important problem for further research is to construct an algorithm that satisfies a \emph{relative} error tolerance, i.e., $\Pr(\abs{p-\hp}/p\leq \varepsilon_r) \geq 1-\alpha$. Such an algorithm would need to find at least a lower bound on $p$.  One would expect the number of samples required then to be proportional to $\var(Y)/(p \varepsilon_r)^2 \sim 1/(p \varepsilon_r^2)$ as $p\varepsilon_r \to 0$.  Our attempts so far at using Hoeffding's inequality results in an algorithm with computational cost proportional to $1/(p \varepsilon_r)^2$ as $p\varepsilon_r \to 0$.  Although, the algorithm for estimating the parameter $p$ to some specified relative error tolerance is successful, the computational cost is extravagant for small $p$. The literature mentioned in the introduction may provide a clue to an algorithm with optimal cost.  If this problem can be solved, then a natural extension would be to construct confidence intervals that satisfy either an absolute or relative error criterion, i.e., confidence intervals of the form $\Pr(\abs{p-\hp}/p\leq \max(\varepsilon_a, \varepsilon_r p)) \geq 1-\alpha$.

\begin{acknowledgement}
This work were partially supported by the National Science Foundation Under grants DMS-1115392 and DMS-1357690.
 The authors would like to express their thanks to the organizers of Eleventh International Conference on Monte Carlo and Quasi-Monte Carlo Methods in Scientific Computing for hosting this great event. The first author would like to express her thanks to  Prof.\ Art B. Owen for arranging generous travel support through the United States National Science Foundation.
\end{acknowledgement}

%
\bibliographystyle{spmpsci}
\bibliography{LJiang}

\begin{thebibliography}{10}
\providecommand{\url}[1]{{#1}}
\providecommand{\urlprefix}{URL }
\expandafter\ifx\csname urlstyle\endcsname\relax
  \providecommand{\doi}[1]{DOI~\discretionary{}{}{}#1}\else
  \providecommand{\doi}{DOI~\discretionary{}{}{}\begingroup
  \urlstyle{rm}\Url}\fi

\bibitem{Agresti02}
Agresti, A.: Categorical Data Analysis.
\newblock wiley-interscience (2002)

\bibitem{Blaker00}
Blaker, H.: Confidence curves and improved exact confidence intervals for
  discrete distributions.
\newblock The Canadian Journal of Statistics \textbf{28}(4), 783--798 (2000)

\bibitem{BS83}
Blyth, C.R., Still, H.A.: Binomial confidence intervals.
\newblock Journal of the American Statistical Association \textbf{78}(381),
  108--116 (1983)

\bibitem{chernoff52}
Chernoff, H.: A measure of asymptotic efficiency for tests of a hypothesis
  based on the sum of observations.
\newblock The Annals of Mathematical Statistics \textbf{23}(4), 493--507 (1952)

\bibitem{GAIL_1_3}
Choi, S.C.T., Ding, Y., Hickernell, F.J., Jiang, L., Zhang, Y.: {GAIL:
  Guaranteed Automatic Integration Library (Version 1.3), MATLAB Software}
  (2014).
\newblock \urlprefix\url{http://code.google.com/p/gail/}

\bibitem{CDHHZ13}
Clancy, N., Ding, Y., Hamilton, C., Hickernell, F.J., Zhang, Y.: The cost of
  deterministic, adaptive, automatic algorithms: Cones, not balls.
\newblock journal of complexity pp. 21--45 (2013)

\bibitem{CP34}
Clopper, C.J., Pearson, E.S.: The use of confidence or fiducial limits
  illustrated in the case of the binomial.
\newblock Biometrika  (1934)

\bibitem{crow56}
Crow, E.: Confidence intervals for a proportion.
\newblock Biometrika  (1956)

\bibitem{HJLO12}
Hickernell, F.J., Jiang, L., Liu, Y., Owen, A.B.: Guaranteed conservative fixed
  width confidence intervals via monte carlo sampling.
\newblock Monte Carlo and Quasi Monte Carlo Methods 2012 pp. 105--128 (2014)

\bibitem{HicJim16a}
Hickernell, F.J., {Jim\'enez Rugama}, {\relax Ll}.A.: Reliable adaptive
  cubature using digital sequences (2014).
\newblock Submitted for publication

\bibitem{H63}
Hoeffding, W.: Probability inequalities for sums of bounded random variables.
\newblock Journal of the American Statistical Association \textbf{58}(301),
  13--30 (1963)

\bibitem{JP04}
Jacod, J., Protter, P.: Probablityy Essentials.
\newblock Springer (2004)

\bibitem{JimHic16a}
{Jim\'enez Rugama}, {\relax Ll}.A., Hickernell, F.J.: Adaptive multidimensional
  integration based on rank-1 lattices (2014).
\newblock In preparation

\bibitem{LB10}
Lin, Z., Bai, Z.: Probability Inequalities.
\newblock Springer (2010)

\bibitem{MATLAB14}
MATLAB: version 8.3.0 (R2014a).
\newblock The MathWorks Inc., Natick, Massachusetts (2014)

\bibitem{sterne54}
Sterne, T.E.: Some remarks on confidence or fiducial limits.
\newblock Biometrika  (1954)

\bibitem{wilson27}
Wilson, E.B.: Probable inference, the law of succession, and statistical
  inference.
\newblock Journal of the American Statistical Association \textbf{22}(158),
  209--212 (1927)

\end{thebibliography}
%
%
%
%
%
%
%
\end{document}